\documentclass[final,12pt]{elsarticle}
\usepackage{a4, amsmath, amssymb, graphics, subfigure, fullpage, color}
\usepackage{soul}
\usepackage[colorlinks]{hyperref}
\usepackage{ulem}
\hypersetup{colorlinks,citecolor=green,filecolor=black,
            linkcolor=blue,urlcolor=blue}
\usepackage{fancyhdr}
\usepackage{multirow}
\usepackage{amsthm}
\usepackage{booktabs}
\newtheorem{theorem}{Theorem}

\newtheorem{corollary}{Corollary}
\newtheorem{proposition}{Proposition}
\newtheorem{lemma}{Lemma}
\newtheorem{example}{Example}
\newtheorem{remark}{Remark}
\newtheorem{assumption}{Assumption}

\newcommand{\cN}{\mathcal{N}}
\newcommand{\cF}{\mathcal{F}}
\newcommand{\FF}{\mathbb{F}}

\DeclareMathOperator{\argmax}{\textrm{argmax}}
\DeclareMathOperator{\E}{\mathbb{E}}

\renewcommand{\Pr}{\mathbb{P}}
\DeclareMathOperator{\var}{\mathbb{V}}

\usepackage{dsfont}
\DeclareMathOperator{\ind}{\mathds{1}}
\DeclareMathOperator{\sign}{\textrm{sign}}
\newcommand{\beq}{\begin{equation}}
\newcommand{\eeq}{\end{equation}}
\newcommand{\MyBib}{MyBib}

\usepackage{lineno}
\begin{document}

\begin{frontmatter}

\title{Bernoulli--Doob representation, asymptotic limits and boundary behavior of bounded homogeneous diffusion martingales}

\author[DB]{Damiano BRIGO}
\address[DB]{Imperial College London, Department of Mathematics,
United Kingdom\\
\href{mailto:damiano.brigo@imperial.ac.uk}
     {damiano.brigo@imperial.ac.uk}}
\author[FV]{Fr\'ed\'eric VRINS}
\address[FV]{UCLouvain, LIDAM/LFIN, Belgium\\
\href{mailto:frederic.vrins@uclouvain.be}
     {frederic.vrins@uclouvain.be}}

\begin{abstract}
We study one-dimensional time-homogeneous diffusion martingales
evolving in a bounded interval $D=[a,b]$, where the diffusion
coefficient vanishes at the endpoints. Under mild regularity
assumptions, we show that such processes admit unique absorbed
strong solutions and are bounded uniformly integrable martingales
(Theorem~\ref{thm:absorbed_existence_uniqueness}). We further prove
that they converge almost surely to a limit supported on $\{a,b\}$,
whose law is the generalized Bernoulli distribution on the two
endpoints with mean equal to the initial value
(Proposition~\ref{prop:variance} and Corollary~\ref{cor:Bernoulli}).

Our main results establish an equivalence, under Brownian-filtration
and Markov assumptions, between this class of bounded homogeneous
diffusion martingales and Bernoulli-Doob martingales
(Theorems~\ref{Th:SDEasCondExp} and~\ref{Th:CondExpasSDE}).
We also clarify the relation between Feller's boundary
classification and the pathwise SDE framework, showing that the
martingale constraint forces absorption at every attainable
boundary. Finally, we prove that the Bernoulli limit is genuinely
asymptotic: for every fixed finite horizon $T$, the probability of
not yet having hit the boundary is strictly positive
(Theorem~\ref{thm:positive_survival_hitting_time}), even when the
individual boundaries are accessible. The results are illustrated by
several explicit examples, including processes arising in credit-risk
modelling.
\end{abstract}


\end{frontmatter}

% ===================================================================
\section{Introduction}
\label{sec:intro}
% ===================================================================

Bounded martingales lie at the intersection of martingale theory,
one-dimensional diffusions, and boundary classification. On the one
hand, boundedness implies uniform integrability and almost sure
convergence. On the other hand, when the martingale is generated by
a time-homogeneous diffusion evolving in a compact interval, the
vanishing of the diffusion coefficient at the endpoints creates a
non-trivial interplay between boundary attainability, pathwise
absorption, and the asymptotic law of the process.

In this paper we study driftless one-dimensional diffusions
\beq
dZ_t = \sigma(Z_t)\,dW_t\,,\qquad Z_0 = z \in [a,b]\,,
\label{eq:dZ}
\eeq
on a compact interval $[a,b]$, under the standing assumptions that
$\sigma(a)=\sigma(b)=0$, $\sigma(x)>0$ for $x\in(a,b)$, and
$\sigma$ is locally H\"older-$1/2$ on $(a,b)$. We ask three basic
questions: what is the limiting law of such bounded homogeneous
diffusion martingales, how is it related to their boundary
classification, and when can such processes be represented as
conditional expectations of endpoint-valued random variables?

Our main results show that the answer is remarkably rigid. Every
diffusion martingale in this class converges almost surely to a
$\{a,b\}$-valued limit whose law is the generalized Bernoulli
distribution on the two endpoints. Moreover, every such process
admits a Bernoulli-Doob representation, and a converse holds for
continuous time-homogeneous Markov martingales on Brownian
filtrations. We also show that the martingale constraint forces
absorption at any attainable boundary, and that the Bernoulli limit
is genuinely asymptotic: for every fixed finite horizon $T$, the
probability of remaining in the interior up to time $T$ is strictly
positive.

While bounded diffusions and bounded martingales arise in several
contexts, general probabilistic results on the asymptotic law,
boundary behavior, and representation of time-homogeneous bounded
diffusion martingales appear to be comparatively limited.  In particular, the Jacobi process is a bounded diffusion which has been used to model stochastic correlations. Its properties have been extensively studied, but are derived under the assumption of non-zero (usually, mean-reverting) drift, thereby ruling out the martingale case from the scope of the analysis; see e.g.~\cite{JenkinsSpano17} and \cite{Markus21} and references therein.\footnote{Some results for the martingale case can be obtained by taking the limit when the mean reversion speed $\kappa$ of the Jacobi process tends to 0. For example, the stationary law of the mean-reverting Jacobi process is a Beta, whose parameters are both proportional to $\kappa$. It is clear that the case $\kappa=0$ is ruled out, because Beta$(0,0)$ is an improper law. However, the family of Beta distributions converges weakly to the Bernoulli distribution as $\kappa\downarrow 0$, such that the latter is expected to be the stationary law of the Jacobi martingale. It is important to note the distinction between the Jacobi process considered in this paper with the \textit{multivariate Jacobi process} (MJP), which is a $K$-dimensional diffusion evolving on the unit simplex: the components of a $K$-dimensional MJP, $Z_{j,t}$, $j\in 1,2,\ldots,K$, are non negative and sum to one at all $t$; see, e.g., \cite{GOURIEROUX2006475}. We refer to~\cite{Lars24} for a detailed discussion of the boundary behavior of hybrid (generalized) MJPs.}%Still, other properties such as the behavior at the boundaries cannot be recovered by computing such a limit.  
Related bounded
martingale examples appear in the conic-martingale literature.
In particular, conic martingales with uniform marginals at all times
are studied in~\cite[Sections 2 \& 3]{BrigoJV20}, but those examples
evolve between time-dependent bounds of the form $[-b(t),b(t)]$ with
$b(0)=0$. Examples of diffusions evolving between constant
boundaries and displaying uniform marginals are given
in~\cite[Section 4]{BrigoJV20}, but the corresponding processes are
mean-reverting and hence are not martingales. %\textcolor{red}{Similarly, the Jacobi process is a widely studied bounded diffusion. However, the corresponding results are derived under the assumption of non-zero (usually, mean-reverting) drift, thereby ruling out the martingale case from the scope of the analysis; see e.g.~\cite{JenkinsSpano17} and \cite{Markus21} and references therein. The stationary law of the mean-reverting Jacobi process is a Beta, whose parameters are proportional to the mean reversion speed $\kappa$. Although Beta$(0,0)$ is not a valid law, the limit of the Beta family as $\kappa\downarrow 0$ is a Bernoulli.} 
Ingersoll~\cite{Ingersoll97} studied
option valuation under a bounded futures foreign-exchange process.\footnote{We are grateful to Bernard Gourion for bringing Ingersoll’s reference to our attention.} The latter can be seen as a ``bounded geometric Brownian motion'' and is a martingale under the appropriate measure. Ingersoll showed that it has a vanishing volatility at the boundaries and analysed the associated
boundary effects in a model-specific setting, see Section \ref{sec:Ingersoll}. 
The $\Phi$-martingale
of Jeanblanc and Vrins~\cite{Vrins16} is a bounded diffusion
martingale which can be thought of as a ``bounded Brownian motion''~\cite{Carr17}, and whose boundaries are proven to be inaccessible. % and whose limit law is a Bernoulli.
Interestingly, the limit laws of the Jacobi, the bounded geometric Brownian motion and the bounded Brownian motion martingales all display a Bernoulli-type behavior, although the formal proof was only given for the last case.

Our contribution is to provide a
general treatment of bounded homogeneous diffusion martingales
within this class. The absorbed diffusions considered here also fit
naturally into the Feller-process framework on compact state spaces;
see, for example, Criens~\cite{Criens23} for recent general results
on Feller--Dynkin processes and martingale problems. We refer to~\cite{Peskir15} for a review of Feller’s work on boundary classification of one-dimensional diffusion processes in the broad context of mathematics and physics.  

Processes of the form~\eqref{eq:dZ} also arise naturally in
mathematical finance, for example in bounded price models, survival
processes, and related conditional-probability martingales. We
return to these connections in Section~\ref{sec:applications}. For
orientation, three representative examples are:
\begin{itemize}
\item the discounted price of a default-free zero-coupon bond in a
      non-negative short-rate model, which is a bounded martingale
      that typically remains in the interior of $[0,1]$;
\item the discounted price of a defaultable zero-recovery bond,
      which may hit the lower boundary before maturity;
\item the survival probability process
      $S^\FF_t(T)=\Pr(\xi>T\mid\cF_t)$, which is a bounded martingale
      and, when $\xi$ is an $\FF$-stopping time, becomes
      $\{0,1\}$-valued after time $T$.
\end{itemize}
These examples illustrate that bounded martingales may behave very
differently near their boundaries, depending on the underlying
information structure.

In addition to the diffusion viewpoint, we establish a connection
with what we call \textit{Bernoulli-Doob martingales}, namely
processes $Z=(Z_t)_{t\ge0}$ of the form
$Z_t=\E[B\mid\cF_t]$, where $B$ is a Bernoulli random variable and
$\FF=(\cF_t)_{t\ge0}$ is a filtration satisfying the usual
conditions. Such processes arise naturally whenever a binary event
is progressively revealed through time.\footnote{This is exactly the
case of the survival probability example above, where
$S^\FF_t(T)=\E[B\mid\cF_t]$ with $B=\ind_{\{\xi>T\}}$.} By
construction, they are bounded in $[0,1]$, have constant
expectation equal to $\Pr(B=1)$, and converge almost surely to a
$\{0,1\}$-valued limit.

We make the following contributions to the literature. 
\begin{enumerate}
\item \textbf{Existence and uniqueness}
      (Theorem~\ref{thm:absorbed_existence_uniqueness}).
      Under mild regularity on $\sigma$, the SDE~\eqref{eq:dZ}
      admits a unique strong absorbed solution on $D=[a,b]$, which
      is a bounded uniformly integrable martingale.
\item \textbf{Bernoulli-Doob representation}
      (Theorems~\ref{Th:SDEasCondExp}
      and~\ref{Th:CondExpasSDE}).
      Every homogeneous diffusion martingale satisfying the above
      conditions admits a Bernoulli-Doob representation, i.e.,
      $Z_t=\E[B\mid\cF_t]$ for $B\sim\mathrm{Ber}(z)$.
      Conversely, every continuous time-homogeneous Markov
      Bernoulli-Doob martingale on a Brownian filtration arises
      from such a diffusion.

\item \textbf{Bernoulli distribution} (Corollary~\ref{cor:Bernoulli}).
The asymptotic Bernoulli behaviour is not confined to specific
examples, but is a general feature of smooth homogeneous
diffusions bounded in $[0,1]$. Moreover, in the Bernoulli--Doob
representation of Theorem~\ref{Th:SDEasCondExp}, the Bernoulli variable can be chosen
to be the terminal limit $B=Z_\infty$, and is therefore
$\cF_\infty$-measurable.

\item \textbf{Truly asymptotic Bernoulli limit}
      (Theorem~\ref{thm:positive_survival_hitting_time}).
      For any fixed $T<\infty$, $\Pr(\tau>T)>0$, where $\tau$ is
      the first hitting time of $\{a,b\}$. This shows that the
      Bernoulli distribution cannot be attained at any finite time,
      even when the individual boundaries are accessible. In other words, in the nondegenerate case $a<z<b$, the Bernoulli
      variable in the Bernoulli--Doob representation cannot be revealed at
      any deterministic finite time.

\item \textbf{Boundary classification}
      (Section~\ref{sec:boundaries}).
      We reconcile Feller's semigroup-based classification with the
      pathwise SDE framework and the martingale constraint, showing
      that the martingale property forces absorption at any
      attainable boundary.
\end{enumerate}

After the first version of this paper appeared on arXiv (14 July 2026, arXiv:2607.12365), Campbell and Engelund posted a related preprint \cite{CampbellEngelund2026} (22 July 2026). Their paper takes a complementary perspective: starting from an autonomous win-martingale, their terminology for our bounded homogeneous diffusion martingale, under the stronger assumption $\sigma \in C^1((0,1))$, they ask which binary sequential experiment generates it as an exact Bayesian posterior. Their main results (Theorems 3.2 and 4.3) provide a constructive embedding via a Lamperti transform and Doob $h$-transforms, and characterise when the posterior is a Markovian function of the current observation through a Riccati equation for the drift gap. These questions are not addressed in our paper. Conversely, our paper establishes the Bernoulli-Doob equivalence under the weaker Hölder-$1/2$ condition on $\sigma$, proves the truly asymptotic nature of the Bernoulli limit (Theorem~\ref{thm:positive_survival_hitting_time}, which has no analogue in their work), and develops the boundary classification reconciling Feller theory with the martingale constraint. The two papers are independent and complementary.

\medskip

The paper is organised as follows.
Section~\ref{sec:setup} introduces the setup and the existence and
uniqueness result.
Section~\ref{sec:boundaries} develops the boundary theory.
Section~\ref{sec:asymptotic} establishes Bernoulli convergence and
the Bernoulli-Doob representation.
Section~\ref{sec:examples} provides several explicit examples.
Section~\ref{sec:applications} briefly discusses connections with
credit-risk modelling.
Section~\ref{sec:BZinfty} provides an example of explicit $B$ for the $Z_t = \E[B|{\cal F}_t]$ representation. 
\ref{App:boundaries} recalls Feller's classification;
\ref{App:4examples} verifies the four examples; and
\ref{app:B} solves the Jacobi moment recursion.

% ===================================================================
\section{Setup, Existence, and Uniqueness}
\label{sec:setup}
% ===================================================================

We consider the driftless SDE~\eqref{eq:dZ} on the compact interval
$D=[a,b]$ under the following standing assumptions, in force
throughout the paper.

\begin{assumption}[Diffusion coefficient]\label{ass:DC} Unless stated otherwise, we consider the following assumptions about the diffusion coefficient in eq. \eqref{eq:dZ}:
\begin{enumerate}
\item[(A1)] $\sigma(a)=\sigma(b)=0$;
\item[(A2)] $\sigma(x)>0$ for all $x\in(a,b)$;
\item[(A3)] $\sigma$ is locally H\"older-$1/2$ continuous on
      $(a,b)$.
\end{enumerate}
\end{assumption}

\begin{remark}[Bounded local martingales are true
martingales]
\label{rem:bounded_mg}
Any solution to~\eqref{eq:dZ} is a continuous local martingale.
Since it is bounded in $[a,b]$, it is automatically a uniformly
integrable true martingale. We use this fact repeatedly below.
\end{remark}

\begin{theorem}[Existence and Uniqueness of the Absorbed Strong
Solution]
\label{thm:absorbed_existence_uniqueness}
Let $D=[a,b]$ be a compact interval and consider~\eqref{eq:dZ}
under Assumption \ref{ass:DC}. % \textnormal{(A1)--(A3)}. 
Define 
\[
\tau_x:=\inf\{t>0:Z_t=x\}\quad\text{and let}\quad \tau:=\tau_a\wedge\tau_b
\]
denote the first hitting time to a boundary. Then:
\begin{enumerate}
\item[\rm(a)] For every $z\in(a,b)$ there exists a unique strong
      solution of~\eqref{eq:dZ} on $[0,\tau)$.
\item[\rm(b)] The stopped process $Z^\tau_t:=Z_{t\wedge\tau}$
      is a strong solution of~\eqref{eq:dZ} on all of $[0,\infty)$.
\item[\rm(c)] If $z\in\{a,b\}$, the unique
      solution to~\eqref{eq:dZ} in $D$ is the constant process
      $Z_t\equiv z$.
\item[\rm(d)] ${Z^\tau}$ is the unique strong solution within
      the class of absorbed solutions on $D$.
\item[\rm(e)] ${Z^\tau}$ is a bounded uniformly integrable
      martingale.
\end{enumerate}
\end{theorem}

\begin{proof}
\textbf{Interior dynamics (proving (a)).}
Let $z\in(a,b)$. For every compact subinterval
$K_\varepsilon=[a+\varepsilon,b-\varepsilon]\subset(a,b)$,
assumption~(A3) implies that $\sigma$ satisfies the
Yamada--Watanabe criterion on $K_\varepsilon$. Hence pathwise
uniqueness holds on each such compact set. Standard localization
arguments yield a unique strong solution up to the first exit time
from $(a,b)$, i.e., up to $\tau$.

\textbf{Boundary initialization (proving (c)).}
Suppose $z=a$. Since $\sigma(a)=0$, the constant process
$Z_t\equiv a$ satisfies the SDE. If $Z$ is any other solution
taking values in $D$, then $Z_t\ge a$ a.s. Every local martingale
bounded from below is a supermartingale (by Remark~\ref{rem:bounded_mg}
applied to $Z-a$), so $\E[Z_t]\le Z_0=a$. Since $Z_t-a\ge0$
a.s., we conclude $Z_t=a$ a.s.\ for all $t\ge0$. The argument for
$z=b$ is identical.

\textbf{Absorbed extension (proving (b)).}
If $\tau=\infty$ there
is nothing to prove. Otherwise, for $t\ge\tau$,
$Z^\tau_t=Z_\tau\in\{a,b\}$, so $\sigma(Z^\tau_t)=0$
by~(A1), and $dZ^\tau_t=0=\sigma(Z^\tau_t)\,dW_t$.
Hence ${Z^\tau}$ satisfies the SDE globally.

\textbf{Uniqueness among absorbed solutions
(proving (d)).}
Let $\widehat{Z}$ be another absorbed solution. Pathwise uniqueness
on $(a,b)$ gives $\widehat{Z}_t=Z^\tau_t$ for $t<\tau$. At
time $\tau$ both processes reach the same boundary point and remain
constant thereafter, so $\widehat{Z}_t=Z^\tau_t$ for all
$t\ge0$ a.s..

\textbf{Martingale property (proving (e)).}
The stopped process ${Z^\tau}$ is a bounded
local martingale. By Remark~\ref{rem:bounded_mg}, it is uniformly
integrable and therefore a true martingale.
\end{proof}

\begin{remark}[Pathwise admissibility versus diffusion startability]
\label{rem:deg}
When we say a boundary point is \emph{not startable}, this is meant
in the classical Feller sense: startability as a non-degenerate
diffusion with a Feller semigroup. This should not be confused with
pathwise admissibility of the SDE. For all examples in this paper,
if $\sigma(a)=0$, then the constant process $Z_t\equiv a$ is a
perfectly valid strong solution even if Feller classifies $a$ as
non-startable (see Theorem~\ref{thm:absorbed_existence_uniqueness}(c)).
Such solutions are absorbing and degenerate; they are not diffusions
if started at the boundary, and are therefore excluded from Feller's
diffusion classification, even though they are entirely acceptable
for our purposes and are included in our classification tables.
\end{remark}

% ===================================================================
\section{Boundary Behaviour of Bounded Martingales}
\label{sec:boundaries}
% ===================================================================

In the previous section, we looked at solutions
to~\eqref{eq:dZ} living in $D=[a,b]$, in the particular case where
$\sigma$ obeys Assumption \ref{ass:DC}. Let us now take a step back and analyse
the SDE~\eqref{eq:dZ} for more general $\sigma$ and $D$. In this
broader context, uniqueness of solutions may be lost, the boundaries
of the process may not coincide with the roots of $\sigma$, and the
behaviour of solutions at the zeros of $\sigma$ depends on the shape
of the diffusion coefficient. 

\subsection{Roots of $\sigma$, state space, and Feller classification}
\label{sec:roots}

We begin with a guiding example that
illustrates these subtleties, then state the general classification.

%\paragraph{Guiding example}
Consider~\eqref{eq:dZ} with $a=0$, $b=1$, and
\beq\label{ex:X1/4}
\sigma(x) = |x(1-x)|^{1/4}.
\eeq
This diffusion coefficient satisfies Assumption~\ref{ass:DC}. In particular, the H\"older exponent at $0$ and $1$ is $1/4<1/2$, proving $(A3)$. The maximal domain of this SDE is $A=\mathbb{R}$, since $\sigma$ is
defined everywhere. If we took $D=A=\mathbb{R}$,
the roots $0$ and $1$ would be interior equilibrium points. Indeed,
since $\sigma(x)=|x(1-x)|^{1/4}$ is H\"older-$1/2$ on both sides
of $0$ and $1$, the Yamada--Watanabe theorem guarantees a unique
strong solution on all of $\mathbb{R}$, and the local time at $0$
and $1$ is zero; the process can therefore cross these points
(see Remark~\ref{rem:nonunique} below). In order to force solutions
to be confined between the zeros, one must impose $D=[0,1]$.
We proceed in four steps.

\begin{enumerate}
\item \textit{Identify the zeros.} The points $x=0$ and $x=1$ are
      roots of $\sigma$. By
      Theorem~\ref{thm:absorbed_existence_uniqueness}(c), if $z=0$
      the unique solution in $D$ is $Z_t\equiv0$; similarly for
      $z=1$.

\item \textit{Define the state space $D=[0,1]$.} We restrict to
      processes confined between these equilibrium points. The
      roots transition from interior points of $A$ to topological
      boundaries of $D$. Imposing this state space
      restricts the set of admissible solutions.

\item \textit{Apply Feller's classification on $D$.} Because $0$
      and $1$ are now topological boundaries, Feller's theory
      applies. Evaluating the relevant integrals
      (~\ref{App:4examples}), both endpoints are classified
      as \emph{regular}.

\item \textit{Impose the martingale constraint.} As established by
      Proposition~\ref{prop:trilemma} below, the martingale
      property forces absorption at any attainable boundary. Of all
      the boundary behaviours Feller allows for regular boundaries
      (reflection, absorption, stickiness), only absorption is
      compatible with the bounded martingale requirement.
\end{enumerate}

A key principle emerges: \textit{a root of $\sigma$ is not
intrinsically a boundary}. If the state space is
not restricted to the interval delimited by the roots, solutions
can cross the zeros of $\sigma$, which then act as interior points
rather than boundaries. Its classification depends on the
definition of the state space. For instance, $\sigma(x)=\sqrt{|x|}$
with $D=\mathbb{R}$ has $x=0$ as an interior point; with
$D=[0,\infty)$ and $\sigma(x)=\sqrt{x}$, $x=0$ is a topological
boundary to which Feller's theory applies. In this paper, when
$\sigma$ has two roots $a<b$ and $z\in[a,b]$, we always take
$D=[a,b]$.

\begin{remark}[Non-uniqueness, local time, and boundary conditions]
\label{rem:nonunique}
The regularity of $\sigma$ at a root $x^*$ determines whether the
solution is unique beyond the first hitting time of $x^*$, and
whether $x^*$ acts as a boundary or an interior point.

\medskip
\noindent\textit{H\"older-$1/2$ on both sides: interior point.}
If $\sigma$ is H\"older-$1/2$ on both sides of $x^*$, the
Yamada--Watanabe theorem guarantees a unique strong solution on all
of $\mathbb{R}$, and $x^*$ is an interior point that can be
crossed. For example, $\sigma(x)=\sqrt{|x|}$ on $\mathbb{R}$ gives
a unique solution for which $x=0$ is interior.

\medskip
\noindent\textit{H\"older-$1/2$ on one side only: boundary.}
If $\sigma$ is defined and H\"older-$1/2$ only on one side of $x^*$
(e.g., $\sigma(x)=\sqrt{x}$ on $[0,\infty)$), then $x^*$ is a
topological boundary. Feller's theory classifies whether it is
attainable.

\medskip
\noindent\textit{Below H\"older-$1/2$: non-uniqueness and local
time.}
When $\sigma$ fails to be H\"older-$1/2$ at $x^*$ (e.g.,
$\sigma(x)=|x|^{1/4}$), pathwise uniqueness is not guaranteed
after the first hitting time of $x^*$. Multiple extensions of the
solution exist, parameterised by the local time coefficient
$\theta$ in
\begin{equation}
\label{eq:XL}
X_t = X_0 + \int_0^t\sigma(X_s)\,dW_s + \theta\,L_t^{x^*}(X)\,,
\end{equation}
where $L_t^{x^*}(X)$ is the local time of $X$ at $x^*$, defined
for any continuous semimartingale by
\[
L_t^{x^*}(X)
= \lim_{\varepsilon\downarrow0}\frac{1}{2\varepsilon}
  \int_0^t\ind_{|X_s-x^*|<\varepsilon}\,d\langle X\rangle_s
= \lim_{\varepsilon\downarrow0}\frac{1}{2\varepsilon}
  \int_0^t\ind_{|X_s-x^*|<\varepsilon}\sigma^2(X_s)\,ds\,.
\]
The four canonical choices are:
\begin{itemize}
\item $\theta=0$: absorption at $x^*$ (the unique solution to the
      SDE itself, since the local time term vanishes);
\item $\theta>0$: reflection at $x^*$;
\item $\theta<0$: $x^*$ is an interior crossing point;
\item $\theta=\infty$: sticky boundary (the process spends positive
      Lebesgue time at $x^*$).
\end{itemize}
The SDE~\eqref{eq:dZ} without additional boundary conditions
corresponds to $\theta=0$, selecting the absorbed solution. The
generator perspective is complementary: the operator
$\mathcal{L}f(x)=\frac{\sigma^2(x)}{2}f''(x)$ is not defined at
roots of $\sigma$, and specifying a boundary condition at $x^*$
corresponds to extending the domain of $\mathcal{L}$. Imposing
$f(x^*)=0$ gives absorption; imposing $f'(x^{*+})=0$ gives
reflection above $x^*$; imposing $f'(x^{*+})=f'(x^{*-})$ gives an
interior point; imposing $\mathcal{L}f(x^*)=kf''(x^*)$ gives a
sticky boundary. The SDE is thus the canonical representation of
the generator with absorption at $x^*$.

As a concrete illustration, the equation
$dX_t=\mathrm{sgn}(X_t)\,dW_t$ has no unique solution: one
solution behaves as a standard Brownian motion (crossing zero),
while others are absorbed at zero. All solutions share the same
quadratic variation $\langle X\rangle_t=t$, but differ in their
behaviour at zero. Imposing the absorbing boundary condition
uniquely selects the martingale solution. For our bounded
martingales, the martingale property independently forces
$\theta=0$, as established by
Proposition~\ref{prop:trilemma}.
\end{remark}

\subsection{The Martingale Trilemma}
\label{sec:trilemma}

The next proposition shows that the boundaries of a $[0,1]$-martingale are necessarily absorbing.

\begin{proposition}[Martingale Trilemma]
\label{prop:trilemma}
Let $Z$ be a continuous process on $[0,1]$ with $\sigma(0)=0$.
The following three conditions cannot hold simultaneously:
\begin{enumerate}
\item[\rm(a)] \textit{Startability:} $\Pr(Z_t>0)>0$ for some $t>0$ when
      $Z_0=0$.
\item[\rm(b)] \textit{Bounded domain:} $Z_t\ge0$ for all $t\ge0$.
\item[\rm(c)] \textit{Martingale property:} $\E[Z_t]=Z_0=0$ for all
      $t\ge0$.
\end{enumerate}
Consequently, within the class of processes satisfying
\textnormal{(b)} and \textnormal{(c)}, the only admissible boundary
behaviour at $0$ is absorption: $Z_t=0$ for all $t\ge\tau_0$ on
$\{\tau_0<\infty\}$.
\end{proposition}

\begin{proof}
Suppose (a) and (b) hold. Then $Z_t>0$ with positive probability
for some $t>0$, and $Z_t\ge0$ a.s., so $\E[Z_t]>0=Z_0$, violating
(c). Suppose (a) and (c) hold. Then $\E[Z_t]=0$ with $Z_t>0$ on a
set of positive probability, which forces $Z_t<0$ on a set of
positive probability, violating (b). Hence (a) must be abandoned:
once $Z$ reaches $0$, it must remain there.
\end{proof}

\begin{remark}\label{rem:CIR} %\textcolor{red}{[I have replaced $X_t$ by $\lambda_t$ here because it allows us to call Remark 4 in the Cox model part in Section 6.]}
%The Trilemma contrasts with the CIR process $dX_t=\kappa(\theta-X_t)\,dt+\eta\sqrt{X_t}\,dW_t$, where the drift $\kappa\theta\,dt>0$ at $X_t=0$ acts as a deterministic engine returning the process to the interior. Our driftless martingales have no such engine: the noise shuts off at the boundary and the process stops.
The Trilemma contrasts with the CIR process $d\lambda_t=\kappa(\theta-\lambda_t)\,dt+\eta\sqrt{\lambda_t}\,dW_t$, where the drift $\kappa\theta\,dt>0$ at $\lambda_t=0$ acts as a deterministic engine returning the process to the interior. Our driftless martingales have no such engine: the noise shuts off at the boundary and the process stops.

Imposing absorption does not violate the martingale property. The
absorbed process $Z_{t\wedge\tau_0}$ is the stopped process; by
Doob's Optional Stopping Theorem, if $Z_t$ is a true martingale in
the interior then $Z_{t\wedge\tau_0}$ is a true martingale for all
$t\ge0$. The expectation is preserved by the law of total
expectation:
\[
\E[Z_{t\wedge\tau_0}]
= \underbrace{\E[Z_t\mid\tau_0\le t]}_{=\,0}\,\Pr(\tau_0\le t)
+ \E[Z_t\mid\tau_0>t]\,\Pr(\tau_0>t) = z\,.
\]
Paths absorbed at $0$ contribute $0$; the surviving paths drift
upward to compensate.
\end{remark}

\subsection{Four boundary classification cases}
\label{sec:fourexamples}
We begin with an example that cements our earlier discussion on the importance of choosing the relevant state space from the start. 
\begin{example}
\label{ex:absval}
The case $\sigma(x)=\sqrt{x(1-x)}$ with $D=[0,1]$ gives a strong
pathwise unique solution with $0$ and $1$ as absorbing (exit)
boundaries. However, $\sigma(x)=\sqrt{|x(1-x)|}$ with $D=\mathbb{R}$
also gives a strong pathwise unique solution, but $0$ and $1$ are
interior points: the process can cross them. The difference is that
$\sqrt{|x(1-x)|}$ is H\"older-$1/2$ on both sides of $0$ and $1$
(the local time at these points is zero), so the process is not
absorbed. The shape of $\sigma$ near its roots, together with the
choice of state space, jointly determine whether the roots act as
boundaries.
\end{example}

Table~\ref{tab:examples-corrected} presents four examples with their
combined Feller and martingale classifications. Detailed
calculations are in~\ref{App:4examples}.

\begin{table}[h!]
\centering
\renewcommand{\arraystretch}{1.3}
\begin{tabular}{|c|c|c|}
\hline
Case $\sigma(x)$ & $x=0$ & $x=1$ \\ \hline
$(x(1-x))^{1/4}$
  & Regular: attainable \& startable
  & Regular: attainable \& startable \\
  & Absorbing
  & Absorbing \\ \hline
$x(1-x)$
  & Natural: inaccessible, not startable
  & Natural: inaccessible, not startable \\
  & ${}^\dagger$
  & ${}^\dagger$ \\ \hline
$\sqrt{x(1-x)}$
  & Exit: attainable, not startable
  & Exit: attainable, not startable \\
  & Absorbing
  & Absorbing \\ \hline
$x^{1/4}(1-x)$
  & Regular: attainable \& startable
  & Natural: inaccessible, not startable \\
  & Absorbing
  & ${}^\dagger$ \\ \hline
\end{tabular}
\caption{Combined Feller and martingale boundary
classification for four diffusion coefficients with $D=[0,1]$,
$\sigma(0)=\sigma(1)=0$. The Feller type (Regular/Exit/Natural) is
determined by the integrability of $\sigma^{-2}$; see~\ref{App:boundaries}. Imposing the martingale property
implies that all attainable boundaries are necessarily absorbing
(Proposition~\ref{prop:trilemma}). If $z\in\{0,1\}$, the unique
solution is $Z_t\equiv z$ for all $t$ (Theorem~\ref{thm:absorbed_existence_uniqueness}(c)).
Entries marked ${}^\dagger$ denote the case where $z\in\{0,1\}$
gives a degenerate constant solution that is pathwise valid but
falls outside Feller's diffusion classification; see
Remark~\ref{rem:deg}.}
\label{tab:examples-corrected}
\end{table}

\begin{remark}[Regular boundaries under the martingale constraint]
\label{rem:regular_clarification}
Table~\ref{tab:examples-corrected} includes Regular boundaries,
which might appear to contradict the statement that attainable
boundaries are always exit type for bounded martingales. The
resolution is as follows. Feller's classification is a property of
the generator and does not depend on the martingale constraint. A
Regular boundary is attainable \emph{and} startable in the Feller
sense. However, Proposition~\ref{prop:trilemma} shows that, once
the martingale property is imposed, the process cannot restart from
a boundary without violating either the domain constraint or the
zero-drift condition. Thus, Regular boundaries are \emph{realized
as absorbing} under the martingale constraint, behaving effectively
like Exit boundaries in terms of pathwise dynamics. The Feller type
describes the analytic properties of the generator; the realized
pathwise behaviour is determined by the combined Feller
classification and the martingale constraint.
\end{remark}

We now discuss each case.

\medskip\noindent\textbf{Case 1} ($\sigma(x)=(x(1-x))^{1/4}$, both
boundaries regular). This case was used earlier to illustrate the importance of clarifying the state space. We keep it here for completeness. Both boundaries are attainable from the
interior and may serve as initial values in the Feller sense.
Since $\sigma(0)=\sigma(1)=0$, the Martingale Trilemma forces
absorption upon hitting.

\medskip\noindent\textbf{Case 2} ($\sigma(x)=x(1-x)$, both
boundaries natural). Neither boundary is attainable from the
interior. An alternative proof uses the fact that $Z$ has the same
dynamics as the logistic transform of $X_t$, where
$dX_t=\frac{\eta^2}{2}\tanh(X_t/2)\,dt+\eta\,dW_t$ (see
Example~3.5 in~\cite{Vrins16}). Therefore
$\{Z_t\in\{0,1\}\}\Leftrightarrow\{X_t\in\{-\infty,+\infty\}\}$,
which has probability zero for all $t\ge0$. This case corresponds to the risk-neutral exchange rate process (bounded Brownian motion) considered in~\cite{Ingersoll97} when $a=0$ and $b=\sigma(t,T)=1$.

\medskip\noindent\textbf{Case 3} ($\sigma(x)=\sqrt{x(1-x)}$,
Jacobi martingale, both boundaries exit). Both boundaries are
attainable with probability one; the martingale requirement forces
absorption upon hitting.

\medskip\noindent\textbf{Case 4} ($\sigma(x)=x^{1/4}(1-x)$,
mixed). The left boundary $0$ is regular (attainable, absorbed by
the martingale constraint). The right boundary $1$ is natural
(inaccessible).

\subsection{Absorbing boundaries}
\label{sec:absorbing}

\begin{lemma}
\label{lemma:ZAbsorbingBoundaries}
Any boundary point reached by a martingale $Z$ bounded in $[0,1]$
is absorbing. More precisely, let
$\tau_b:=\inf\{t\ge0:Z_t=b\}$, $b\in\{0,1\}$. Then, on
$\{\tau_b<\infty\}$, $Z_t=b$ for all $t\ge\tau_b$. Moreover, if
$Z$ is a time-homogeneous diffusion with coefficient $\sigma(Z_t)$,
then $\sigma(b)=0$.
\end{lemma}

\begin{proof}
Focus on $b=1$; the case $b=0$ is analogous. For all $t\ge\tau_1$,
$Z_t\le1$ a.s. By the Optional Stopping Theorem,
$\E[Z_t\mid\cF_{\tau_1}]=Z_{\tau_1}=1$. Hence $Z_t$ is a bounded
random variable whose conditional expectation given $\cF_{\tau_1}$
equals its upper bound $1$, forcing $Z_t=1$ a.s.\ on
$\{\tau_1<\infty\}$ for all $t\ge\tau_1$. If $\sigma(1)\ne0$, then
$d\langle Z\rangle_t=\sigma^2(1)\,dt>0$, contradicting the fact
that $Z_t$ is stuck at $1$.
\end{proof}

% ===================================================================
\section{Asymptotic Law and Bernoulli-Doob Representation}
\label{sec:asymptotic}
% ===================================================================

In the previous section, we discussed the
possibility for a process $Z$ solving~\eqref{eq:dZ} to reach the
boundaries. In this section, we investigate the asymptotic law of
$Z$. We find that the limit is a Bernoulli, independently of the
boundary classification. Moreover, we establish the equivalence
between such processes and Bernoulli-Doob martingales.

\subsection{Variance-based proof of Bernoulli convergence}
\label{sec:variance}

\begin{remark}[Fokker-Planck heuristic]
\label{rem:FP}
Before the rigorous proof, we offer an informal but illuminating
argument via the Fokker-Planck equation. This requires the
additional assumption $\sigma\in C^2$ and a formal interchange of
limit and differentiation that we do not justify here; the rigorous
proof follows in Proposition~\ref{prop:variance} and
Corollary~\ref{cor:Bernoulli} below.

Suppose $Z$ admits a stationary density in the sense that
$\lim_{t\to\infty}p(t,x)=p^\star(x)$ for some generalised density
$p^\star$. The Fokker-Planck equation reads
\[
\frac{\partial p(t,x)}{\partial t}
= \frac{\partial^2[\sigma^2(x)p(t,x)]}{\partial x^2}\,.
\]
Setting the left-hand side to zero gives
$\sigma^2(x)p^\star(x)=k_1 x+k_2$ for constants $k_1,k_2$. Since
$p^\star\ge0$ and $\sigma^2(a)=\sigma^2(b)=0$, we need
$k_1 a+k_2=0$ and $k_1 b+k_2=0$, which forces $k_1=k_2=0$. Hence
$p^\star(x)=0$ for all $x\in(a,b)$, so $p^\star$ is supported
entirely on $\{a,b\}$:
\[
p^\star(x) = k\,\delta(x-a) + (1-k)\,\delta(x-b)\,.
\]
The martingale condition $\int x\,p^\star(x)\,dx=z$ gives
$k=1-(z-a)/(b-a)$, which is precisely the generalised Bernoulli
$\mathcal{B}(a,b,z)$. The intuition is transparent: the only affine
function that is non-negative and vanishes at both $a$ and $b$ is
the zero function, so the stationary density must be supported
entirely on the boundary.
\end{remark}

\begin{proposition}[Variance characterization of Bernoulli
convergence]
\label{prop:variance}
Under assumptions \textnormal{(A1)--(A3)}, let
$v(t):=\var[Z_t]$.
\begin{enumerate}
\item[\rm(i)] \textnormal{(Max-variance law.)}
      The generalised Bernoulli $\mathcal{B}(a,b,z)$ is the unique
      maximum-variance distribution on $[a,b]$ with mean $z$:
      \[
      \mathcal{B}(a,b,z)
      = \argmax_{\mathcal{L}\in\mathcal{P}(a,b,z)}\var[\mathcal{L}]\,,
      \]
      where $\mathcal{P}(a,b,z)$ is the set of distributions on
      $[a,b]$ with mean $z$.
\item[\rm(ii)] \textnormal{(Variance strictly increasing.)}
      $v(t)$ is strictly increasing as long as
      $\Pr(Z_t\in\{a,b\})<1$.
\end{enumerate}
\end{proposition}

\begin{proof}
\textbf{(i)} We prove for $a=0$, $b=1$; the general case follows
by shifting and scaling. Let $X$ have law
$\mathcal{L}_X\in\mathcal{P}(0,1,z)$ and
$Y\sim\mathcal{B}(0,1,z)$, so $\E[X]=\E[Y]=z$. Then
\[
\var[X]
= \int_0^1 x^2\,\mu(dx) - z^2
\le \int_0^1 x\,\mu(dx) - z^2
= z(1-z) = \var[Y]\,,
\]
since $x^2\le x$ on $[0,1]$, with equality if and only if
$x\in\{0,1\}$. Equality holds if and only if $\mu$
assigns all mass to $\{0,1\}$ with mass $z$ at $1$, i.e., if and
only if $\mathcal{L}_X=\mathcal{B}(0,1,z)$.

\textbf{(ii)} By It\^o's isometry,
\[
v(t) = \int_0^t\E[\sigma^2(Z_s)]\,ds\,.
\]
The integrand is strictly positive unless $\Pr(\sigma^2(Z_s)=0)=1$,
i.e., unless $Z_s\sim\mathcal{B}(a,b,z)$ (in which case paths are
frozen at the bounds for all $s'\ge s$). Hence $v$ increases
strictly until the distribution becomes Bernoulli.
\end{proof}

\begin{corollary}[Almost sure Bernoulli convergence]
\label{cor:Bernoulli}
Under Assumption~\ref{ass:DC}, $Z_t\to\{a,b\}$ a.s.\
as $t\to\infty$, and the limit law is $\mathcal{B}(a,b,z)$.
\end{corollary}

\begin{proof}
Since $v(t)$ is strictly increasing and bounded above by
$(b-z)(z-a)$, we have $\E[\sigma^2(Z_t)]\to0$ as $t\to\infty$.
This is equivalent to $\Pr(Z_t\in\{a,b\})\to1$. Since $Z$ is a
bounded martingale, it converges a.s.\ to an integrable limit
$Z_\infty$ (see e.g.\ \cite{Cohen15}, Th.~5.2.1). By the above,
$Z_\infty\in\{a,b\}$ a.s., and $\E[Z_\infty]=z$ gives the Bernoulli
weights.
\end{proof}

\subsection{Bernoulli-Doob representation theorems}
\label{sec:representation}

\begin{theorem}[Bounded homogeneous diffusions are Bernoulli-Doob
martingales]
\label{Th:SDEasCondExp}
Let $(\Omega,\cF,\FF,\Pr)$ satisfy the usual conditions with
$\FF=(\cF_t)_{t\ge0}$ and $\cF=\cF_\infty$. Let $W$ be a Brownian
motion and $Z=(Z_t)_{t\ge0}$ take values in $[0,1]$ and solve
\[
dZ_t = \sigma(Z_t)\,dW_t\,,\qquad Z_0=z\in[0,1]\,,
\]
where $\sigma:[0,1]\to[0,\infty)$ is continuous, strictly positive
on $(0,1)$, and satisfies $\sigma(0)=\sigma(1)=0$. Then $Z$ is a
continuous Bernoulli-Doob martingale:
\[
Z_t = \E[B\mid\cF_t]\,,\qquad B\sim\mathrm{Ber}(z)\,,\quad t\ge0\,.
\]
\end{theorem}

\begin{proof}
By Remark~\ref{rem:bounded_mg}, $Z$ is a true
martingale. By the martingale convergence theorem, there exists
$Z_\infty\in L^1$ with $Z_t\to Z_\infty$ a.s.\ and in $L^1$
\cite[Problem~3.19]{Kara05}. Moreover \cite[Problem~3.20]{Kara05},
there exists $X\in L^1$ with $Z_t=\E[X\mid\cF_t]$ and
$\E[X\mid\cF]=Z_\infty$ a.s. The tower law gives
$Z_t=\E[Z_\infty\mid\cF_t]$. By Corollary~\ref{cor:Bernoulli} with $(a,b)=(0,1)$, $Z_\infty\in\{0,1\}$ a.s. Setting $B:=Z_\infty$ and noting $\E[B]=\E[Z_\infty]=z$, we conclude $B\sim\mathrm{Ber}(z)$.
\end{proof}

\begin{theorem}[Characterization by Bernoulli variables]
\label{Th:CondExpasSDE}
Let $(\Omega,\cF,\FF,\Pr)$ satisfy the usual conditions with
$\cF=\cF_\infty$. Define $Z_t=\E[B\mid\cF_t]$ where
$B\sim\mathrm{Ber}(z)$. Then $Z$ is an $\FF$-martingale bounded
in $[0,1]$ with $Z_0=z$.

Moreover, suppose $\FF$ is a Brownian filtration and $Z$ is a
continuous time-homogeneous Markov process. Then $Z$
can be written as a homogeneous diffusion
\[
dZ_t = \sigma(Z_t)\,dW_t\quad\text{where}\quad Z_0=z
\]
for some $\sigma:[0,1]\to[0,\infty)$ and an $\FF$-Brownian motion
$W$. Furthermore, $\sigma(b)=0$ at any boundary $b\in\{0,1\}$
reachable in finite time.
\end{theorem}

\begin{proof}
Since $B$ is integrable and $0\le B\le1$, $Z$ is an
$\FF$-martingale with $0\le Z_t\le1$ and $Z_0=\E[B]=z$.

Assume $Z$ is a continuous Markov process and $\FF$ is a Brownian
filtration. Then $Z$ is a continuous square-integrable martingale
adapted to $\FF$. By the martingale representation theorem, there
exist an $\FF$-Brownian motion $W'$ and an
$\FF$-predictable process $H$ with
\[
Z_t = Z_0 + \int_0^t H_s\,dW'_s\,,\qquad
\langle Z\rangle_t = \int_0^t H_s^2\,ds\,.
\]
By L\'evy's characterization, $W_t:=\int_0^t\sign(H_s)\,dW'_s$
is an $\FF$-Brownian motion and $dZ_t=|H_t|\,dW_t$.

For a one-dimensional continuous time-homogeneous Markov process,
the quadratic variation is a homogeneous additive functional of the
process. By the theory of additive functionals
\cite[Ch.~V]{RogersWilliams00}, there exists a Borel function
$a:[0,1]\to[0,\infty)$ such that
$\langle Z\rangle_t=\int_0^ta(Z_s)\,ds$, i.e.,
$|H_t|=\sqrt{a(Z_t)}$ for $dt\times d\Pr$-a.e.\ $(t,\omega)$.
Setting $\sigma(x):=\sqrt{a(x)}$ gives
$dZ_t=\sigma(Z_t)\,dW_t$.

The last claim follows from Lemma~\ref{lemma:ZAbsorbingBoundaries}.
\end{proof}

\begin{remark}
By the same argument as in Theorem~\ref{Th:SDEasCondExp}, the
Bernoulli variable $B$ in Theorem~\ref{Th:CondExpasSDE} coincides
a.s.\ with $Z_\infty$. 
Without time-homogeneity, the martingale representation need not reduce to a coefficient depending only on the current state.
\end{remark}

\subsection{The Bernoulli limit is truly asymptotic}
\label{sec:asymptotic_limit}

\begin{theorem}
\label{thm:positive_survival_hitting_time}
Consider an $[a,b]$-martingale $Z$ solving~\eqref{eq:dZ} with
$a<z<b$ and $\sigma$ bounded on $[a,b]$.\footnote{Boundedness of
$\sigma$ on $[a,b]$ is automatic under~(A1)--(A3) since $\sigma$ is
continuous on the compact set $[a,b]$.} Then
\[
\Pr(\tau>t)>0\qquad\forall\,t>0\,.
\]
\end{theorem}

\begin{proof}
Since $a<Z_0<b$, $\Pr(\tau>0)=1$. Since $\sigma$ is bounded, there
exists $M>0$ with $|\sigma(x)|\le M$ on $[a,b]$. The process
$N_t:=Z_{t\wedge\tau}-z$ is a continuous martingale with $N_0=0$
and bounded quadratic variation
\beq
\langle N\rangle_t
= \int_0^{t\wedge\tau}\sigma(Z_s)^2\,ds \le M^2 t\,.
\label{eq:QVNs}
\eeq
By the Dambis-Dubins-Schwarz theorem, there exists a Brownian
motion $B=(B_s)_{s\ge0}$ with $N_t=B_{\Theta_t}$, where
$\Theta_t=\langle N\rangle_t$. Let $d:=\min\{z-a,b-z\}>0$ and
\[
A_t := \Bigl\{\sup_{0\le s\le M^2 t}|B_s|<d\Bigr\}\,.
\]
From~\eqref{eq:QVNs}, $\Theta_s\in[0,M^2 t]$ for $s\le t$. On
$A_t$, $|B_{\Theta_s}|<d$ for all $s\le t$, so $|N_s|<d$ and
$Z_{s\wedge\tau}\in(a,b)$ for all $s\le t$. Therefore
$A_t\subseteq\{\tau>t\}$, giving $\Pr(\tau>t)\ge\Pr(A_t)>0$ since
$B$ has strictly positive probability of remaining in $(-d,d)$ over
any finite interval.
\end{proof}

\begin{corollary}
\label{cor:asymptotic}
For any $T<\infty$, $\Pr(Z_T\notin\{a,b\})>0$. Even when the
individual boundaries are accessible (as in the Jacobi case), the
event $\{\tau\le T\}$ has strictly positive probability for all
$T>0$ but never equals $1$ for any finite $T$. For $[0,1]$-martingales, the Bernoulli
distribution acts as a genuine asymptotic limit, not a finite-time
distribution.
\end{corollary}

% ===================================================================
\section{Some Tractable Examples}
\label{sec:examples}
% ===================================================================

We illustrate the theory with four tractable
examples for $(a,b)=(0,1)$. The first two are homogeneous diffusion
cases, chosen to represent the two qualitatively different boundary
regimes: natural (inaccessible) and exit (accessible). The last two
highlight the importance of the time-homogeneity assumption in
Theorems~\ref{Th:SDEasCondExp}
and~\ref{thm:positive_survival_hitting_time}.

\subsection{The $\Phi$-martingale}
\label{sec:phi}
%$$
%Z_t=Z_0\;\exp\left\{-\frac{\sigma^2}{2} t+\sigma W_t\right\}
%$$
%$$
%Z_t:=\Phi\left(e^{\frac{\eta^2}{2}t}(\Phi^{-1}(Z_0)+\eta\int_0^te^{-\frac{\eta^2}{2}s}\,dW_s)\right)
%$$

The $\Phi$-martingale \cite{Vrins16} with constant volatility $\eta$
corresponds to the process
\[
Z_t := \Phi(X_t)\,,\quad
X_t := e^{\frac{\eta^2}{2}t}(k+I_t)\,,\quad
I_t := \eta\int_0^te^{-\frac{\eta^2}{2}s}\,dW_s\,,\quad
k = \Phi^{-1}(z)\,,
\]
where $\Phi$ is the standard normal CDF and $\phi=\Phi'$. As studied independently by Carr~\cite{Carr17}, this process can be thought of as the natural analogue of Brownian motion
on the compact set $[0,1]$. It is easy to check from It\^o's
lemma that $Z$ solves the time-homogeneous SDE \eqref{eq:dZ} with%
\footnote{The function $\phi(\Phi^{-1}(x))$ is locally
H\"older-$1/2$, provable via Mills' formula; see
e.g.~\cite{Dumbgen10}. Hence $Z$ is a particular case of
Theorem~\ref{thm:absorbed_existence_uniqueness}.}
\[
%dZ_t = \sigma(Z_t)\,dW_t\quad\text{where}\quad
\sigma(x) = \eta\phi(\Phi^{-1}(x))\,,\quad z\in[0,1]\,.
\]

Both boundaries are natural (inaccessible). The variance is
\[
v(t) = \Phi_2\!\left(\Phi^{-1}(z),\Phi^{-1}(z);\rho(t)\right)-z^2\,,
\qquad\rho(t)=1-e^{-\eta^2 t}\,,
\]
where $\Phi_2$ is the bivariate standard normal CDF. As
$t\to\infty$, $\rho(t)\to1$ and $v(t)\to z(1-z)$, confirming
Bernoulli convergence. Since $Z$ is the image of a Gaussian process
under $\Phi$, the boundaries are never reached in finite time,
i.e., $\Pr(\tau>t)=1$ for all $t\ge0$ when
$z\in(0,1)$, consistent with the natural boundary classification.

The $\Phi$-martingale admits the Bernoulli-Doob representation with
$B:=Z_\infty=\ind_{\{k+I_\infty>0\}}\sim\mathrm{Ber}(z)$,
confirming Theorem~\ref{Th:SDEasCondExp}. To see this, set
$\mathcal{Z}\sim\cN(0,1)$ independently of $\cF_t$. Then,
\begin{align*}
\E[\ind_{\{k+I_\infty>0\}}\mid\cF_t]
&= \Pr\!\left(\eta^2\sqrt{\int_t^\infty e^{-\eta^2 s}\,ds}
   \cdot \mathcal{Z} > -(k+I_t)\,\Bigg|\,\cF_t\right)
 = \Phi\!\left(\frac{k+I_t}{\sqrt{e^{-\eta^2 t}}}\right)
 = \Phi(X_t) = Z_t\,.
\end{align*}

\subsection{The Jacobi martingale}
\label{sec:jacobi}

The Jacobi martingale corresponds to~\eqref{eq:dZ} with
$\sigma(x)=\sqrt{\eta x(1-x)}$, which is H\"older-$1/2$ on $[0,1]$. This process is also known as the (driftless version of the) Wright-Fisher process~\cite{JenkinsSpano17}. This is related to the process we encountered earlier but in the special case $\eta=1$. Both boundaries are of exit type (attainable, absorbing). The
second moment $g(t):=\E[Z_t^2]$ satisfies the ODE
$\dot{g}=\eta z-\eta g$ with $g(0)=z^2$, giving
\[
v(t) = z(1-z)(1-e^{-\eta t})\,,
\]
which converges to $z(1-z)$ as $t\to\infty$, confirming Bernoulli
convergence.

The expected first hitting time to either boundary, starting from
$Z_0=x$, satisfies $\mathcal{L}m=-1$ with $m(0)=m(1)=0$, where
$\mathcal{L}f(x)=\frac{\eta}{2}x(1-x)f''(x)$. This gives
\[
\E_x[\tau] = -\frac{2}{\eta}
\bigl(x\ln x+(1-x)\ln(1-x)\bigr)\,,
\]
which is finite for all $x\in(0,1)$, so $\Pr_x(\tau<\infty)=1$.
The conditional expected hitting times are
\[
\E_x[\tau_0]
= \E_x[\tau\mid Z_\tau=0]
= -\frac{2}{\eta}\frac{x\ln x}{1-x}\,,
\]
and $\E_x[\tau_1]$ is obtained by replacing $x$ by $1-x$.

By Theorem~\ref{thm:positive_survival_hitting_time}, despite
$\Pr_x(\tau<\infty)=1$, we have $\Pr_x(\tau>T)>0$ for all
$T<\infty$. The Bernoulli limit is truly asymptotic.

\begin{remark}\label{Rem:BJacobi} Because $Z$ is a martingale, one gets $Z_t=\E[Z_\infty\mid\cF_t]$ where $Z_\infty:=\lim_{t\to\infty} Z_t$, $Z_t:=z+\int_0^t\sigma(Z_s)dW_s$. One can check that in the specific Jacobi case with $\eta=1$, i.e., when $\sigma(z)=\sqrt{z(1-z)}$, $Z_\infty(\omega)=\ind_{\{\tau_1(\omega)<\tau_0(\omega)\}}$ a.s.. This provides the explicit expression of the binary random variable $B$ in our Bernoulli-Doob representation result, Theorem~\ref{Th:SDEasCondExp}. See also the case in Section~\ref{sec:BZinfty} below. 
\end{remark}

\subsection{A time-inhomogeneous case}
\label{sec:inhomogeneous}

It is easy to check using It\^o's lemma that the
process $Z=(Z_t)_{t\ge0}$ defined as
$Z_t=\Phi\!\left(X_t/\sqrt{1+e^{\eta^2 t}}\right)$, where $X$ is
as in Section~\ref{sec:phi} with $k=\Phi^{-1}(z)\sqrt{2}$, %\footnote{\textcolor{red}{\textbf{Note: the correct expression is $k=\Phi^{-1}(z)\sqrt{2}$, not $k=\Phi(z)\sqrt{2}$. The argument of $\Phi^{-1}$ is the initial value $z\in(0,1)$, and $\Phi^{-1}$ maps $(0,1)$ to $\mathbb{R}$. Using $\Phi(z)$ instead would give a value in $(0,1)$ and would not produce the correct initial condition $Z_0=z$.}}}
solves the SDE
\[
dZ_t = \sigma(t,Z_t)\,dW_t\quad\text{where}\quad
\sigma(t,x) := \frac{\phi(\Phi^{-1}(x))}{\sqrt{1+e^{\eta^2 t}}}\,,
\quad Z_0=z\in[0,1]\,.
\]
Interestingly, $Z$ can also be written as the conditional
expectation of a random variable $Z_\infty$ bounded in $[0,1]$.
Indeed, setting $\mathcal{Z}\sim\cN(0,1)$ independently of $\cF_t$:
\begin{align*}
\E[\Phi(k+I_\infty)\mid\cF_t]
&= \E\!\left[\Phi\!\left(k+I_t+\eta^2\sqrt{\int_t^\infty
   e^{-\eta^2 s}\,ds}\cdot\mathcal{Z}\right)\right]
 = \Phi\!\left(\frac{k+I_t}{\sqrt{1+e^{-\eta^2 t}}}\right)
 = Z_t\,.
\end{align*}
This shows that $Z$ is a true martingale in $[0,1]$ and can be
written as the conditional expectation of
$Z_\infty=\Phi(k+I_\infty)$, which is bounded in $[0,1]$. However,
in contrast with the $\Phi$-martingale, $Z_\infty$ is \emph{not}
Bernoulli distributed: $Z_\infty=\Phi(k+I_\infty)\sim\Phi(k+\mathcal{Z})$
with $\mathcal{Z}\sim\cN(0,1)$, which is a continuous distribution
on $(0,1)$. This does not contradict Theorem~\ref{Th:SDEasCondExp}
because the diffusion coefficient is not time-homogeneous: compared
to the $\Phi$-martingale, the denominator $\sqrt{1+e^{\eta^2 t}}$
in $\sigma(t,x)$ grows without bound, reducing the volatility of
the process along each trajectory and preventing it from diverging
to the bounds. This explains why the asymptotic behaviour is no
longer Bernoulli.

\subsection{A time-changed $[0,1]$-martingale}
\label{sec:timechange}

Theorem~\ref{thm:positive_survival_hitting_time}
shows that no homogeneous diffusion martingale can hit the bounds
almost surely before a fixed finite time $T$. However,
time-inhomogeneous diffusions can achieve $\Pr(\tau\le T)=1$. One
way to reproduce this behaviour is to use an exploding time-change
of a homogeneous martingale, as we now show.\medskip

Given a homogeneous $[0,1]$-martingale $Z$ solving \eqref{eq:dZ}, define the time-changed process
$Y_u:=Z_{t(u)}$ where
\[
u(t) = T(1-e^{-\lambda t})
\quad\Leftrightarrow\quad
t(u) = -\frac{1}{\lambda}\log\!\left(1-\frac{u}{T}\right)
\]
for some $\lambda>0$. As $t$ ranges $[0,\infty)$, $u$ ranges
$[0,T)$. The dynamics of $Y$ are
\[
dY_u = \sqrt{(\lambda(T-u))^{-1}}\,\sigma(Y_u)\,d\tilde{W}_u\,,
\qquad u\in[0,T)\,,
\]
where $\tilde{W}$ is a standard Brownian motion under the
time-changed filtration. Since $Y_T=Z_\infty\in\{a,b\}$ a.s., the
process hits one of the bounds before $T$ almost surely.
This does not contradict
Theorem~\ref{thm:positive_survival_hitting_time} because the
diffusion coefficient is time-inhomogeneous. For $u>T$, one may
extend by replacing $u/T$ with $(u\wedge T)/T$, after which $Y$
remains constant at the attained bound.

In the special case where the latent process $Z$ is
the $\Phi$-martingale discussed in Section \ref{sec:phi}, the diffusion coefficient of $Y$ is similar
to that of the process $Z$ in Section~\ref{sec:inhomogeneous}: the
only difference is the denominator, which changes from
$\sqrt{1+e^{\eta^2 t}}$ to $\sqrt{\lambda(T-t)}$. In contrast with
the former, the latter vanishes as $t\to T$, causing $\sigma$ to
explode at $t=T$, and explains why $Y_T\in\{0,1\}$ almost surely.

%Figure~\ref{fig:plots} shows sample paths of $Y$ for several
%choices of $\sigma$ and $\lambda$.\textcolor{red}{[I suggest we move the figures in a supplementary material annex.]}

\ifdefined\undefined{
\begin{figure}[h]
\centering
\subfigure[$\lambda=1$, $\sigma^2(x)=x^2(1-x)^2$]{%
  \includegraphics[width=0.45\columnwidth]
  {Fig/PlotSigmaX1mX-Lambda1.pdf}}\hspace{0.2cm}
\subfigure[$\lambda=2$, $\sigma^2(x)=x^2(1-x)^2$]{%
  \includegraphics[width=0.45\columnwidth]
  {Fig/PlotSigmaX1mX-Lambda2.pdf}}\\
\subfigure[$\lambda=1$, $\sigma^2(x)=\sqrt{x(1-x)}$]{%
  \includegraphics[width=0.45\columnwidth]
  {Fig/PlotSigmaSqrt0d25X1mX-Lambda1.pdf}}\hspace{0.2cm}
\subfigure[$\lambda=2$, $\sigma^2(x)=\sqrt{x(1-x)}$]{%
  \includegraphics[width=0.45\columnwidth]
  {Fig/PlotSigmaSqrt0d25X1mX-Lambda2.pdf}}\\
\subfigure[$\lambda=1$, $\sigma^2(x)=\sqrt{x}(1-x)^2$]{%
  \includegraphics[width=0.45\columnwidth]
  {Fig/PlotSigmaSqrt0d25XTimesX1mX-Lambda1.pdf}}\hspace{0.2cm}
\subfigure[$\lambda=2$, $\sigma^2(x)=\sqrt{x}(1-x)^2$]{%
  \includegraphics[width=0.45\columnwidth]
  {Fig/PlotSigmaSqrt0d25XTimesX1mX-Lambda2.pdf}}
\caption{Ten sample paths of $Y$ with parameters
$\Delta t=10^{-4}$, $T=2$, $Y_0=1/2$. The dotted lines mark the
boundaries $\{0,1\}$. All paths converge to a boundary by time $T$,
consistent with $\Pr(\tau^Y\le T)=1$.}
\label{fig:plots}
\end{figure}
}\fi

% ===================================================================
\section{Applications to Credit Risk and Financial Modelling}
\label{sec:applications}
% ===================================================================

%\textcolor{red}{[Many changes in this section. The subsections now have the same type of structure and I provided extra results that help understand the differences and why there is no contradiction with our theory. Please check. In addition, I have changed the default time from $\tau$ to $\xi$, to avoid the notation clash with $\tau:=\tau_0\wedge \tau_1$. The parts in red are the changes which require more attention.]}.\medskip

\subsection{Credit Risk Modelling} 

The Bernoulli-Doob representation arises naturally in credit
risk modelling, where $Z_t=\E[B\mid \cF_t]$ coincides with $S^\FF_t(T)$ when $B=\ind_{\{\xi>T\}}$. For conciseness, we restrict ourselves to provide three examples associated with the main families of default models.\medskip

The default time $\xi$ of a reference
entity is a stopping time under the enlarged filtration
$\mathbb{G}=(\mathcal{G}_t)_{t\ge0}$, where
$\mathcal{G}_t=\cF_t\vee\sigma(\{\{\xi<u\},u\le t\})$. In this case, the survival probability process $S^\mathbb{G}(T)$ defined as $S^\mathbb{G}_t(T):=\Pr(\xi>T\mid\mathcal{G}_t)$ is a $\mathbb{G}$-martingale in $[0,1]$ converging to $0$ or $1$ as $t\to T$, and hitting $0$ before $T$ whenever $\xi<T$. This may not be the case for $S^\FF(T)$, unless $\xi$ is an $\FF$-stopping time. The following default models show how various information
structures lead to different boundary types for the $\FF$-martingale
$Z=S^\FF(T)$, which is known to play a key role in default models~\cite{Biel02}. All these results are consistent with the Bernoulli-Doob framework exposed above. Table~\ref{tab:credit} summarises the three cases. We note $z=Z_0=S_0^\FF(T)$.

\medskip\noindent\textbf{Structural model.}
In standard structural models, the Bernoulli variable takes the form $B=\ind_{\{V_T>K\}}$ where the firm-value process $V$ and the debt level $K$ are $\FF$-adapted, such that $\xi$ is an $\FF$-stopping time. Therefore, $\{\xi>T\}\in\cF_T$, hence,
$Z_t=S^\FF_t(T)=S_t^\mathbb{G}(T)$ for all $t\geq 0$, $Z_t\in\{0,1\}$ for all $t\ge T$, and
$\Pr(\tau\le T)=1$. In this case, the Bernoulli law is not only reached asymptotically as $t\to\infty$, but $Z_t\sim Ber(z)$ exactly, for all $t\geq T$.\medskip

In the original Merton's model \cite{Merton74}, where $K$ is constant and $V$ is an $\FF$-GBM with parameters $(\mu,\sigma)$,
\[
Z_t = \Phi(d^+_t)\quad\text{where}\quad
d^+_t := \frac{(\mu-\frac{\sigma^2}{2})(T-t)+\ln(V_t/K)}
             {\sigma\sqrt{T-t}}\,.
\]

From It\^o's lemma, the dynamics of $Z$ read as
\beq
dZ_t=\ind_{\{t<T\}}\frac{\phi(\Phi^{-1}(Z_t))}{\sqrt{T-t}}\,dW_t,\label{eq:dZ-Struct}
\eeq
i.e., the diffusion coefficient is time-inhomogeneous on $[0,T)$.  
This makes clear why $Z$ reaches the bounds almost surely by time $T$ without contradicting Theorem~~\ref{thm:positive_survival_hitting_time}: the diffusion coefficient is explicitly time-inhomogeneous and becomes singular as $t\uparrow T$. Likewise, there is no contradiction with Theorem~\ref{Th:CondExpasSDE}. Although $Z$ is a Bernoulli-Doob martingale, its quadratic variation density is
\[
\frac{d\langle Z\rangle_t}{dt}
=
\ind_{\{t<T\}}\frac{\phi(\Phi^{-1}(Z_t))^2}{T-t},
\]
which depends explicitly on time and not only on the current state $Z_t$. Hence $Z$ is not a one-dimensional time-homogeneous diffusion of the form $dZ_t = \sigma(Z_t) dW_t$.
Interestingly, the diffusion coefficient of $Z$ in the structural model coincides with that of the time-changed process $Y$ considered in Section~\ref{sec:timechange} when $\lambda=1$ and the latent
process $Z$ therein follows the $\Phi$-martingale dynamics discussed in Section \ref{sec:phi}.\medskip

\medskip\noindent\textbf{Cox model.}
In Cox models, the Bernoulli variable reads as $B=\ind_{\{\int_0^T\lambda_s\,ds<\mathcal{E}\}}$, where $\lambda$ is $\FF$-adapted and
$\mathcal{E}\perp\cF_\infty$, such that $\xi$ is not an $\FF$-stopping time: $\{\xi>T\}\notin\cF_\infty$. Therefore, $Z=S^\FF(T)$ never reaches the bounds
(natural boundaries). In such models, $Z$ is not Markovian (it depends on both
$\lambda_t$ and $\Lambda_t:=\int_0^t\lambda_s\,ds$), so
Theorem~\ref{Th:CondExpasSDE} does not apply and explains why the Bernoulli-Doob martingale $Z$ may not take the form of a time-inhomogeneous diffusion.\medskip

Observe that, as in Merton's model, the limit law is not purely asymptotic, but is reached a.s. at the finite time $T$ because $Z_t=\Pr(\mathcal{E}\ge \Lambda_T|\Lambda_T)$ for all $t\ge T$. However, in contrast with Merton, $Z_t$ is not Bernoulli distributed after $T$ because $\xi$ is not an $\FF$-stopping time. For instance, if $\lambda$ is the CIR %with volatility $\eta$ driven by a Brownian motion $W$, as 
considered in Remark~\ref{rem:CIR} and $\mathcal{E}\sim Exp(1)$, it holds that $Z_t=\exp\{-\Lambda_T\}\not\sim Ber(z)$ for all $t\geq T$. This does not contradict Corollary~\ref{cor:Bernoulli} because the diffusion coefficient of $Z$ does not meet Assumption \ref{ass:DC} on $[0,T)$. Indeed, it is not time-homogeneous:
\beq
dZ_t=-\ind_{\{t<T\}}\eta B(T-t)Z_t\sqrt{\lambda_t}dW_t,\label{eq:dZ-Cox}
\eeq
where $s\to B(s)$ is the usual function associated with the log-affine form 
\[
\E[\exp\{-\Lambda_T\}|\cF_t]= A(T-t^T)\exp\{-\Lambda_{t^T}-B(T-t^T)\lambda_T\}\quad\text{where}\quad t^T:=t\wedge T.
\]
In fact, $\sigma$ in \eqref{eq:dZ-Cox} is not even of the form $\sigma(t,Z_t)$.

\medskip\noindent\textbf{$\Phi$-martingale model.} In such models, the Bernoulli variable of interest is $B=\ind_{\{\Phi^{-1}(z)+I_\infty>0\}}$ where the Brownian motion $W$ governing the stochastic integral $I$ is $\FF$-adapted. We refer to~\cite{Vrins26} for a discussion about how such models can be applied to price collateralized debt obligation (CDO) options and credit valuation adjustment (CVA) under wrong-way risk.

Therefore, $\{\xi>T\}\in\cF_\infty$ but $\{\xi>T\}\notin\cF_t$ for all
$t<\infty$. The boundaries of $Z=S^\FF(T)$ are natural and $Z$ is time-homogeneous. All assumptions of
Theorem~\ref{Th:CondExpasSDE} are met, giving the $\Phi$-martingale
of Section~\ref{sec:phi}.\medskip

Both the Cox and the $\Phi$-martingale models belong to the
reduced-form family, but they differ in that the Az\'ema
supermartingale $S^\FF_t(t)$ is non-increasing in the former,
but features a martingale part in the latter. This specificity explains why the $\Phi$-martingale model behaves very differently from Cox. To see this, recall that the $\Phi$-martingale model yields
\beq
dZ_t=\eta\phi(\Phi^{-1}(Z_t))dW_t,\label{eq:dZ-Phi}
\eeq
i.e., the corresponding diffusion coefficient is \emph{not} killed as of $t=T$, which is a major difference with the above two families of models: the indicator $\ind_{\{t<T\}}$ appears in \eqref{eq:dZ-Struct} and \eqref{eq:dZ-Cox} but not in \eqref{eq:dZ-Phi}. Economically speaking, this means that the agent cannot observe $\xi$ in $\FF$, as in Cox. However, in contrast with Cox, the likelihood of the event $\{\xi>T\}$ is continuously updated after $T$ according to the information collected in $\FF$. In particular, the so-called \textit{immersion property} (also known as \textit{$H$ hypothesis}) holds in Cox but not in the $\Phi$-martingale case; see
\cite{Vrins16} and \cite{Vrins26} for details.

\begin{table}[h!]
\centering
\renewcommand{\arraystretch}{1.3}
\begin{tabular}{lcccc}
\toprule
Model & $B\in\cF_\infty$? & Boundaries & $Z$ Markov? &
Applicable result \\
\midrule
Structural
  & Yes
  & Exit
  & No
  & \\%Section~\ref{sec:timechange} \\
(Merton)
  & ($B\in\cF_T$)
  & (accessible)
  & (inhomogeneous)
  & Section~\ref{sec:timechange} \\
  \hline
Cox
  & No
  & Natural
  & No
  &\\% --- \\
  & ($B\notin\cF_t$, $\forall t\le \infty$)
  & (inaccessible)
  & No
  & --- \\
  \hline
$\Phi$-martingale
  & Yes
  & Natural
  & Yes
  & \\%Thm.~\ref{Th:CondExpasSDE} \\
  
  & ($B\notin\cF_t$, $\forall t<\infty$)
  & (inaccessible)
  & Yes
  & Theorem~\ref{Th:CondExpasSDE} \\
\bottomrule
\end{tabular}
\caption{Summary of the three default time models. The column
``$B\in\cF_\infty$?'' indicates whether the Bernoulli variable
$B=\ind_{\{\xi>T\}}$ is measurable with respect to the market
filtration at infinity. Theorem~\ref{Th:CondExpasSDE} applies when
$Z$ is a continuous time-homogeneous Markov process on a Brownian
filtration.}
\label{tab:credit}
\end{table}

\subsection{FX Modelling}\label{sec:Ingersoll} 
Financial applications outside the scope of credit and interest rate risks exist as well. For instance, Ingersoll argues that monetary policies usually aim to keep the currency values within a certain range, proposing to model the maturity-$T$ futures foreign exchange rate $Z$ as a martingale living in $[a,b]$ under the appropriate measure~\cite{Ingersoll97}. The relationship with our work, more precisely, is that it is shown in~\cite[Prop. 2]{Ingersoll97} that the risk-neutral futures exchange rate $Z$ cannot reach the bounds if the diffusion coefficient takes the form $(Z_{t}-a)(b-Z_{t})\sigma(t,T)$ where $\sigma(t,T)$ is uniformly bounded. As noted earlier, this is essentially the second case of boundary classification we examined. The author also stresses in Figure~1 of that article that the density of $Z$ becomes bimodal after some time when $g(t,T)=1$ and indicates that ``the probability that the exchange rate will be found close to one of the two barriers tends to accumulate'', but does not discuss the limit behavior any further nor does he consider alternative martingales. In this sense, that work, from a mathematical point of view, is a qualitative and numerical precursor for one case of our exact results.

% ===================================================================

\section{Explicit $B$ in the Jacobi case}
\label{sec:BZinfty}

In the next lemma, we provide the explicit indicator form of the variable $B$ in the representation $Z_t=\E[B\mid\cF_t]$ where $Z$ is the Jacobi martingale, as explained in Remark~\ref{Rem:BJacobi}.

\begin{lemma}[Jacobi case: explicit Bernoulli variable]
\label{lem:JacobiConstructiveB}
Let $Z=(Z_t)_{t\ge0}$ be the Jacobi martingale solving
\[
dZ_t=\sqrt{\eta Z_t(1-Z_t)}\,dW_t,\qquad Z_0=z\in[0,1],
\]
and define
\[
\tau_0:=\inf\{t\ge0: Z_t=0\},\qquad
\tau_1:=\inf\{t\ge0: Z_t=1\},\qquad
\tau:=\tau_0\wedge\tau_1.
\]
Set
\begin{equation}
B:=\ind_{\{\tau_1<\tau_0\}}.\label{eq:Bjacobi}
\end{equation}
Then:
\begin{enumerate}
\item[(i)] $B\in\{0,1\}$ almost surely;
\item[(ii)] $B=Z_\infty$ almost surely, where $Z_\infty:=\lim_{t\to\infty}Z_t$;
\item[(iii)] for every $t\ge0$,
\[
Z_t=\E[B\mid\cF_t];
\]
\item[(iv)] in particular,
\[
\Pr_z(\tau_1<\tau_0)=z,
\]
so that $B\sim\mathrm{Ber}(z)$.
\end{enumerate}
\end{lemma}

\begin{proof}
The cases $z=0$ and $z=1$ are trivial, so assume $z\in(0,1)$. In the Jacobi case, the process $Z$ is a bounded continuous martingale on $[0,1]$, both boundaries are absorbing once hit, and
\[
\Pr_z(\tau<\infty)=1.
\]
Since the paths are continuous and the boundary points are distinct, one cannot have $\tau_0=\tau_1<\infty$. Hence the events $\{\tau_1<\tau_0\}$ and $\{\tau_0<\tau_1\}$ are disjoint and form an almost sure partition of the sample space. Therefore
\[
B=\ind_{\{\tau_1<\tau_0\}}\in\{0,1\}
\qquad\text{a.s.,}
\]
which proves (i).

Because $Z$ is a bounded martingale, it is uniformly integrable, and the martingale convergence theorem yields an integrable random variable $Z_\infty$ such that
\[
Z_t\to Z_\infty
\qquad\text{a.s. and in }L^1.
\]
On the event $\{\tau_1<\tau_0\}$, absorption at $1$ gives
\[
Z_t=1,\qquad t\ge\tau_1,
\]
hence $Z_\infty=1$ there. Similarly, on $\{\tau_0<\tau_1\}$ one has
\[
Z_t=0,\qquad t\ge\tau_0,
\]
so $Z_\infty=0$ there. Since these two events form an almost sure partition, it follows that
\[
Z_\infty=\ind_{\{\tau_1<\tau_0\}}=B
\qquad\text{a.s.,}
\]
which proves (ii).

Since $Z$ is uniformly integrable and converges in $L^1$ to $Z_\infty$, one has
\[
Z_t=\E[Z_\infty\mid\cF_t],\qquad t\ge0.
\]
Using (ii), this becomes
\[
Z_t=\E[B\mid\cF_t],\qquad t\ge0,
\]
which proves (iii).

Finally, evaluating at $t=0$ gives
\[
z=Z_0=\E[B]=\Pr(B=1)=\Pr_z(\tau_1<\tau_0),
\]
so $B\sim\mathrm{Ber}(z)$. This proves (iv).
\end{proof}

\begin{remark}
The representation of the Bernoulli variable given in eq.~\eqref{eq:Bjacobi} is not specific to the Jacobi martingale case we illustrate in this section. %More generally, for any bounded homogeneous diffusion martingale on $[0,1]$ such that both boundaries are attainable from the interior, one has
%$Z_\infty=\ind_{\{\tau_1<\tau_0\}}$ a.s. 
%%Indeed, $\tau:=\tau_0\wedge\tau_1<\infty$ almost surely in that case, and attainable boundaries are absorbing under the martingale constraint. The Jacobi case is a particularly explicit example.
%
%To see this, observe that $B=\tilde{B}$ a.s. where $\tilde{B}(\omega):=\ind_{\{\tau_1(\omega)=\tau(\omega)\}}$. By definition, $\tau=\tau_0\wedge\tau_1$, hence, $\tilde{B}$ is 1 if the boundary that $Z$ hits first before $t=\infty$ is one and 0 if the boundary that is hit first by $Z$ before $t=\infty$ is zero. %Next, define $S:=\{\omega\in\Omega:Z_\infty(\omega)\neq \tilde{B}(\omega)\}$ where $Z_\infty\in\{0,1\}$ a.s.. 
%Because the boundaries of $Z$ are absorbing and the first hitting time is finite, $Z_\infty=Z_\tau$ is 1 if $\tau=\tau_1$ and 0 if $\tau=\tau_0$. This shows that $Z_\infty=\tilde{B}=B$ a.s. 
More generally, let $Z$ be a bounded homogeneous diffusion martingale on $[0,1]$ such that both boundaries are attainable from the interior, and define
\[
\tau_0:=\inf\{t\ge 0: Z_t=0\}, \qquad
\tau_1:=\inf\{t\ge 0: Z_t=1\}, \qquad
\tau:=\tau_0\wedge\tau_1.
\]
Then $\tau<\infty$ a.s., and
\[
Z_\infty=\ind_{\{\tau_1<\tau_0\}} \qquad \text{a.s.}
\]
Indeed, by continuity of the paths one cannot have $\tau_0=\tau_1<\infty$, so almost surely exactly one of the events $\{\tau_1<\tau_0\}$ and $\{\tau_0<\tau_1\}$ occurs. Since attainable boundaries are absorbing under the martingale constraint, on $\{\tau_1<\tau_0\}$ one has $Z_t=1$ for all $t\ge \tau_1$, while on $\{\tau_0<\tau_1\}$ one has $Z_t=0$ for all $t\ge \tau_0$. Hence $Z_\infty=1$ on $\{\tau_1<\tau_0\}$ and $Z_\infty=0$ on $\{\tau_0<\tau_1\}$, proving the claim.
\end{remark}

% ===================================================================
% CONCLUSION
% ===================================================================

\section{Conclusion}

%\textcolor{red}{[New section.]}\medskip

%\input{./Submitted Version/Conclusion (old).tex}

%\textcolor{red}{[I propose a slightly more ``technical'' conclusion as an alternative (the former conclusion is in the file ``conclusion (old).tex''). I think we can afford this type of conclusion in a probability paper, but this is just an idea. Just think about it. Note that I make the distinction between the convergence of the law (which is what we show first) and only then the convergence of the random variables. Seen like this, the Bernoulli-Doob representation looks rather obvious: 1) we known that $Z_t\to Z_\infty$ 2) the martingale property of $Z$ yields $Z_t=\E[Z_\infty|\cF_t]$ 3) our distribution results say that $Z_\infty$ must be Bernoulli distributed, hence 4)  $Z$ admits the Bernoulli-Doob representation, it suffices to take $B=Z_\infty$.]}

This paper studies bounded one-dimensional martingales $Z$ generated by
time-homogeneous driftless diffusions on a compact interval, i.e.,
\[
Z_t=z+\int_0^t\sigma(Z_s)dW_s\;,\quad \Pr(Z_t\in[a,b])=1~~\forall t\geq 0.
\]

We show that when the diffusion coefficient $\sigma$ satisfies Assumption~\ref{ass:DC}, the distribution of $Z_t$ as $t\to\infty$ converges to $\mathcal{B}(a,b,z)$, a generalized Bernoulli distribution putting a mass of $z$ (resp. $1-z$) on the value $b$ (resp. $a$). Moreover, this distribution cannot be reached in finite time: $Z_t\not\sim \mathcal{B}(a,b,z)$ for all $t<\infty$, unless $z\in\{a,b\}$, in which case $Z_t=z$ is degenerated. This proves that the Bernoulli limit is genuinely asymptotic: the probability of remaining in the interior up to any fixed finite time is strictly positive. In addition, this conclusion holds even when the boundaries are accessible. In particular, it is independent of the Feller type of the boundaries. We further show that attainable boundaries
must be absorbing under the martingale constraint, thereby
reconciling Feller's boundary classification with the pathwise SDE
framework.

A second main contribution is the link with Bernoulli-Doob
martingales: every martingale $Z$ in our
class admits a Bernoulli-Doob representation, i.e.,
\[
Z_t=\E[B\mid\cF_t]\quad\text{where}\quad B\sim\mathcal{B}(a,b,z)\;.
\]
This stems from the fact that $Z_t$ converges almost surely to some random variable $Z_\infty$. Using the martingale property of $Z$, the above equality holds with $B=Z_\infty$. Finally, we know from our previous results that the latter is a generalized Bernoulli. A converse results holds
for continuous time-homogeneous Markov martingales on Brownian
filtrations. %We also prove that the Bernoulli limit is genuinely asymptotic: even when the individual boundaries are accessible, the probability of remaining in the interior up to any fixed finite time is strictly positive.

Several explicit examples illustrate the different possible boundary
regimes and the role of time-homogeneity. Finally, the paper
indicates how the Bernoulli-Doob viewpoint connects naturally with
survival processes and related models from the credit risk literature.

%\textcolor{red}{[We should perhaps finish the conclusion with a striking point, eg, explaining why our results matter. Here is an idea, but we could probably say more.] An important practical implication of our results it that the Bernoulli convergence that has been highlighted (empirically or theoretically) for specific bounded diffusion martingales such as the Jacobi martingale and the ``bounded extensions'' of the Brownian motion and geometric Brownian motion is, in fact, unavoidable. In particular, the only [?] way to design bounded diffusion martingales with non-Bernoulli limit is to consider time-inhomogenous diffusion coefficients.}

% ===================================================================
% BIBLIOGRAPHY
% ===================================================================
\newpage
\ifdefined\MyBib
  \bibliography{\MyBib}
  \bibliographystyle{plain}
\fi

\newpage
% ===================================================================
\appendix
% ===================================================================

% ===================================================================
\section{Feller Boundary Classification}
\label{App:boundaries}
% ===================================================================

This appendix recalls Feller's boundary classification for
one-dimensional diffusions in natural scale, specialised to the
driftless setting of this paper. The standard references are Karlin
and Taylor \cite[pp.~228--250]{KaTa85}, Borodin and Salminen
\cite[pp.~14--15]{Boro02}, and Rogers and Williams
\cite{RogersWilliams00}.

\begin{proposition}[Feller boundary classification]
\label{prop:boundary-classification}
Consider $dZ_t=\sigma(Z_t)\,dW_t$ on $[a,b]$ with $Z_0=z\in(a,b)$,
$\sigma\in C^2([a,b])$, $\sigma>0$ on $(a,b)$,
$\sigma(a)=\sigma(b)=0$. Since the drift is zero, the diffusion is
in natural scale $s(x)=x$ with speed measure density
$m(x)=2/\sigma^2(x)$. Fix an interior reference point $z\in(a,b)$
and define
\[
N_a := \int_a^z(z-x)\sigma^{-2}(x)\,dx\,,\qquad
\Sigma_a := \int_a^z(x-a)\sigma^{-2}(x)\,dx\,,
\]
and analogously $N_b$, $\Sigma_b$ at the right boundary $b$. The
boundary type is determined by the following table.
\[
\renewcommand{\arraystretch}{1.4}
\begin{array}{c|c|l}
N_a & \Sigma_a & \text{Boundary type of }a \\ \hline
<\infty & <\infty & \text{Regular: attainable and startable} \\
<\infty & =\infty & \text{Exit: attainable, not startable} \\
=\infty & <\infty & \text{Entrance: not attainable, startable} \\
=\infty & =\infty & \text{Natural: not attainable, not startable}
\end{array}
\]
Here $N_a<\infty$ means $\Pr_x(\tau_a<\infty)>0$ for all
$x\in(a,b)$; $N_a=\infty$ means $\tau_a=\infty$ a.s.
$\Sigma_a<\infty$ means $a$ is an admissible starting point in the
diffusion sense; $\Sigma_a=\infty$ means it is not. The same table
applies to the right boundary $b$.
\end{proposition}

% ===================================================================
\section{Four Examples of Boundary Classification}
\label{App:4examples}
% ===================================================================

We verify the four cases of Table~\ref{tab:examples-corrected}
using Proposition~\ref{prop:boundary-classification}. For each
case we record the local asymptotics of $\sigma^{-2}$ near each
endpoint and determine the finiteness of $(\Sigma,N)$.

\subsection*{Case 1: $\sigma(x)=(x(1-x))^{1/4}$ (both endpoints
regular)}

$\sigma^{-2}(x)=1/\sqrt{x(1-x)}$. Near $0$:
$(z-x)\sigma^{-2}(x)=\mathcal{O}(x^{-1/2})$ and
$x\sigma^{-2}(x)=\mathcal{O}(x^{1/2})$, so both $\Sigma_0$ and
$N_0$ are finite. By symmetry the same holds at $1$. Both endpoints
are \emph{regular}.

Closed-form primitive: using the substitution $x=\sin^2\theta$,
\[
\int\frac{x-z}{\sqrt{x(1-x)}}\,dx
= (1-2z)\arcsin\sqrt{x}-\sqrt{x(1-x)}+C\,.
\]
Both endpoint limits are finite, confirming
$(\Sigma_0,N_0)=(\text{finite},\text{finite})$ and
$(\Sigma_1,N_1)=(\text{finite},\text{finite})$.

\subsection*{Case 2: $\sigma(x)=x(1-x)$ (both endpoints natural)}

$\sigma^{-2}(x)=1/[x^2(1-x)^2]$. Near $0$:
$(z-x)\sigma^{-2}(x)\sim Cx^{-2}$ and
$x\sigma^{-2}(x)\sim x^{-1}$, so $\Sigma_0=N_0=\infty$. By
symmetry both endpoints are \emph{natural}.

\subsection*{Case 3: $\sigma(x)=\sqrt{x(1-x)}$ (both endpoints
exit)}

$\sigma^{-2}(x)=1/[x(1-x)]$. Near $0$:
$(z-x)\sigma^{-2}(x)\sim zx^{-1}$ and $x\sigma^{-2}(x)\sim1$, so
$\Sigma_0=\infty$ and $N_0<\infty$. By symmetry both endpoints are
\emph{exit}. Since $\sigma(0)=\sigma(1)=0$, the process is absorbed
upon hitting.

\subsection*{Case 4: $\sigma(x)=x^{1/4}(1-x)$ (left regular, right
natural)}

$\sigma^{-2}(x)=x^{-1/2}(1-x)^{-2}$. Near $0$:
$\Sigma_0<\infty$ and $N_0<\infty$ (\emph{regular}). Near $1$:
$(x-z)\sigma^{-2}(x)\sim(1-x)^{-2}$ and
$(1-x)\sigma^{-2}(x)\sim(1-x)^{-1}$, so $\Sigma_1=N_1=\infty$
(\emph{natural}).

% ===================================================================
\section{Moment Recursion for the Jacobi Martingale}
\label{app:B}
% ===================================================================

For the Jacobi martingale $\sigma(x)=\sqrt{\eta x(1-x)}$, the
conditional moments $M_n(t):=\E[Z_t^n\mid Z_s=z]$ satisfy a
closed-form recursion that provides an independent verification of
Bernoulli convergence (Corollary~\ref{cor:Bernoulli}) for this
specific model, and also matches the results of~\cite{Gour02} in
the limit $\kappa\to0$.

Let $\lambda_n:=\eta n(n-1)/2$ with $\lambda_0=\lambda_1=0$.

\begin{theorem}
\label{thm:momentrecursion}
Suppose $(M_n(t))_{n\ge0}$ satisfies for $n\ge1$ the integral
recursion
\[
M_n(t) = z^ne^{-\lambda_n(t-s)}
+\lambda_n\int_s^te^{-\lambda_n(t-u)}M_{n-1}(u)\,du\,,
\]
with $M_0\equiv1$ and $M_1\equiv z$. Then:
\begin{enumerate}
\item[\rm(1)] \textbf{(ODE)}
      $\dfrac{d}{dt}M_n(t)=\lambda_n(M_{n-1}(t)-M_n(t))$,
      $M_n(s)=z^n$.
\item[\rm(2)] \textbf{(Matrix exponential)}
      With $B$ the lower-bidiagonal matrix $B_{n,n}=-\lambda_n$,
      $B_{n,n-1}=\lambda_n$:
      $M(t)=\exp((t-s)B)\,M(s)$.
\item[\rm(3)] \textbf{(Closed form)}
      For $0\le m\le n$:
      \[
      M_n(t)
      = \sum_{m=0}^nz^m\sum_{j=m}^ne^{-\lambda_j(t-s)}
        \frac{\displaystyle\prod_{r=m+1}^n\lambda_r}
             {\displaystyle\prod_{\substack{r=m\\r\ne j}}^n
              (\lambda_j-\lambda_r)}\,.
      \]
\item[\rm(4)] \textbf{(Examples)}
      $M_0=1$, $M_1=z$,
      $M_2(t)=z+z(z-1)e^{-\eta(t-s)}$,
      \[
      M_3(t)=z+\tfrac{3}{2}z(z-1)e^{-\eta(t-s)}
      +\bigl(z^3-\tfrac{3}{2}z(z-1)-z\bigr)e^{-3\eta(t-s)}\,.
      \]
\end{enumerate}
As $t\to\infty$, $M_n(t)\to z$ for all $n\ge1$, confirming that
$Z_\infty\sim\mathrm{Ber}(z)$.
\end{theorem}

\begin{proof}
\textbf{(1)} Differentiate the integral recursion using Leibniz'
rule; the integral term cancels exactly, leaving the stated ODE.
Evaluating at $t=s$ gives $M_n(s)=z^n$.

\textbf{(2)} The ODE system $\frac{d}{dt}M(t)=BM(t)$ has unique
solution $M(t)=\exp((t-s)B)M(s)$.

\textbf{(3)} For finite $N$, the $(N+1)\times(N+1)$ truncation
$B_N$ is lower-bidiagonal with explicitly computable matrix
exponential entries (obtained by diagonalising each $2\times2$
block or by repeated variation of constants). The formula
stabilises as $N\to\infty$ since each entry depends only on
finitely many $\lambda_j$.

\textbf{(4)} For $n=2$: $\frac{d}{dt}M_2=\eta(z-M_2)$,
$M_2(s)=z^2$, giving $M_2(t)=z+(z^2-z)e^{-\eta(t-s)}$. For
$n=3$: substitute $M_2$ and integrate.

The convergence $M_n(t)\to z$ follows since $e^{-\lambda_j(t-s)}\to0$
for all $j\ge2$ (as $\lambda_j>0$ for $j\ge2$), leaving only the
constant term $z$.
\end{proof}

\end{document}